\documentclass[11pt]{article}
\usepackage{amsmath,amssymb,amsthm}
\usepackage[margin=1.2in]{geometry}
\usepackage[hidelinks]{hyperref}

\newtheorem{theorem}{Theorem}
\newtheorem{lemma}[theorem]{Lemma}
\newtheorem{proposition}[theorem]{Proposition}
\newtheorem{corollary}[theorem]{Corollary}
\theoremstyle{definition}
\newtheorem{definition}[theorem]{Definition}
\newtheorem{remark}[theorem]{Remark}

\newcommand{\Z}{\mathbb{Z}}
\newcommand{\N}{\mathbb{N}}
\newcommand{\SA}{S_{\mathrm{A}}}
\newcommand{\SB}{S_{\mathrm{B}}}
\newcommand{\pos}{\mathrm{pos}}
\newcommand{\OG}{\mathrm{OG}}

\title{Structural rigidity in the Erd\H{o}s--Graham two-set permutation problem}
\author{William Kasel}
\date{August 2026}

\begin{document}
\maketitle

\begin{abstract}
Call a permutation of a set of positive integers \emph{admissible} if it
contains no increasing or decreasing three-term arithmetic progression in
position order. Davis, Entringer, Graham, and Simmons proved that the
positive integers are not admissibly permutable, partitioned them into
three admissibly permutable sets, and asked whether two sets suffice. We
study the canonical dyadic candidate
$\SA = \bigcup_{k \ge 2,\ k \text{ even}} (2^{k-1}, 2^{k}]$. We establish
several structural restrictions on admissible permutations, including an
orbit obstruction, an asymmetric balance law, and the impossibility of
placing all sufficiently large dyadic blocks as contiguous runs. For the
main result, we reduce admissibility of $\SA$ to feasibility of an order
gadget on $(M, 2M]$: the order must avoid monotone three-term progressions
and satisfy fifteen precedence constraints induced by the values $15$ and
$16$. We then prove that a three-constraint core of this gadget is
inconsistent for every $M \equiv 0 \pmod 8$, using zigzag propagation on
arithmetic-progression ladders, a transfer lock, and mirror-flood
induction. Since every relevant dyadic scale lies in this residue class,
$\SA$ admits no admissible permutation. Thus the canonical dyadic partition
does not solve the two-set problem, which remains open. Independent SAT
encodings and certificate checks provide auxiliary verification of the
human-readable proof.
\end{abstract}

\section{Introduction}

Call a linear arrangement of a set of integers \emph{admissible} if no three
of its terms, read in order of position, form an increasing or a decreasing
arithmetic progression (a \emph{monotone 3-AP}). Every finite set of integers
admits an admissible arrangement: place the odd values before the even values
and recurse, using the fact that the first and third terms of a 3-AP share a
parity (stated for intervals in \cite{DEGS77}, which credits earlier notes of
this observation; for arbitrary finite sets in \cite{ErGr79}). For infinite sets the situation changes completely.
Davis, Entringer, Graham and Simmons \cite{DEGS77} proved that every
permutation of $\Z^+$ (of order type $\omega$) contains an increasing 3-AP ---
in particular a monotone one --- so $\Z^+$ itself admits no admissible
permutation. They also constructed a partition of $\Z^+$ into \emph{three}
sets, each admitting an admissible permutation, and asked (Concluding
Remarks~3, p.~89) whether two sets suffice. Erd\H{o}s and Graham repeated the
question in \cite{ErGr79} (p.~338), calling it ``a very annoying question''
(see also their monograph \cite{ErGr80}); it is problem \#197 of the
Erd\H{o}s problem collection \cite{Bl}, where it is stated for the natural
numbers and carries a Lean formalization \cite{FC} (see
Remark~\ref{rem:lean} for the $0 \in \N_0$ nuance).

The modern literature around the problem concerns densities, longer
progressions, and enumeration (Ho \cite{Ho26} recently proved that
$\lim_n \theta(n)^{1/n}$ does not exist, where $\theta(n)$ counts the
$3$-AP-free permutations of $[n]$, answering a question of Landman and
Robertson). LeSaulnier and Vijay \cite{LV11} exhibited a set admitting an
admissible permutation with upper density $1/2$ and lower density $1/4$;
writing $\alpha_{\N}(3)$ (resp.\ $\beta_{\N}(3)$) for the supremum of upper
(resp.\ lower) densities of such sets, this gives $\alpha_{\N}(3) \ge 1/2$
and $\beta_{\N}(3) \ge 1/4$, and they conjectured both bounds sharp.
Geneson \cite{Gen26} recently disproved the upper-density conjecture, proving
$\alpha_{\N}(3) \ge 2/3$. His witness is strikingly relevant here: it is a
union of ratio-$4$ doubling blocks $[L\cdot 4^j + m,\, 2L\cdot 4^j]$ --- an
$\SA$-like octave pattern \emph{with the bottom sliver of each block
removed} (here $m = M_{k-1}$ in his notation, re-cased to avoid a collision
with our scale parameter $M$) and only finitely many octaves per stage. Our
main theorem shows the full dyadic blocks fail, and the mechanism of failure
--- the attack of the values $15$ and $16$ on every block's bottom sliver
(Sections~\ref{sec:og}--\ref{sec:c3}) --- is concentrated in exactly the
region Geneson's construction excises, so the two results are consistent
and their mechanisms align. (We do not claim that
removing precisely those slivers is necessary or sufficient; his
construction also uses stage separation and rapid scale growth.) Permutations of $\Z$ and of subsets of $\Z^+$
avoiding longer monotone progressions were constructed by Adenwalla
\cite{Ad24} and Geneson \cite{Gen18}. Hirose and Saito \cite{HS24} ---
continuing a line begun by Ardal, Brown and Jungi\'c \cite{ABJ11} on
AP-avoiding (``chaotic'') orderings of $\mathbb{Q}$ and $\mathbb{R}$ ---
characterized the \emph{order types} of total orders on $\N$, $\Z$ and $\mathbb{Q}$
admitting no monotone 3-AP; from their standpoint the difficulty of \#197 is
that order type $\omega$ is forced, where avoidance is hardest.

A natural first candidate for a two-set partition is the \emph{dyadic
partition}: let
\[
\SA \;=\; \bigcup_{k \text{ even},\, k \ge 2} \left(2^{k-1},\, 2^{k}\right]
\;=\; (2,4] \cup (8,16] \cup (32,64] \cup \cdots,
\]
\[
\SB{} \;=\; \Z^+ \setminus \SA \;=\; \{1,2\} \cup (4,8] \cup (16,32] \cup \cdots,
\]
so that each team is a union of doubling blocks (together, for $\SB{}$, with
the two initial values) and the completion
$2y - x$ of a within-block pair $x < y$, whenever it escapes the block
upward, lands in the adjacent block --- the other team's territory. Two
elementary density facts explain why this is the canonical candidate. First,
$\SA$ has upper density $2/3$ (attained along $n = 4^k$) and lower density
$1/3$ (along $n = 2\cdot 4^k$), and $\SB{}$ the same with the roles of the two
subsequences exchanged. So $\SA$ sits exactly at the current record value
$\alpha_{\N}(3) \ge 2/3$ of \cite{Gen26}, and the main theorem below
therefore exhibits a specific set \emph{at} that density which is not
permutable --- the record is not improved by taking full dyadic blocks.
Second, for any partition $\Z^+ = A \sqcup B$ the identity
$\overline{d}(A) + \underline{d}(B) = 1$ holds automatically, so a two-set
solution requires $\alpha_{\N}(3) + \beta_{\N}(3) \ge 1$ (a remark of
\cite{LV11}, who conjectured the sum to be $3/4 < 1$); the dyadic pair meets
this constraint exactly, $2/3 + 1/3 = 1$, while the best bounds currently
known give only $\alpha_{\N}(3) + \beta_{\N}(3) \ge 2/3 + 1/4 = 11/12$, so
the density obstruction is undecided. Our main result eliminates this
candidate on other grounds.

\begin{quote}
\textbf{Main Theorem} (Theorem~\ref{thm:main}). \emph{$\SA$ admits no
permutation of order type $\omega$ free of monotone $3$-term arithmetic
progressions. Consequently the dyadic partition does not witness a positive
answer to problem \#197.}
\end{quote}

The proof chain is short and we state it first, since the paper contains a
good deal of surrounding material that is not part of it. It has exactly
two links. The first is a reduction (Theorem~\ref{thm:ogred}): an
\emph{order gadget} $\OG(M)$ on the interval $(M, 2M]$ --- AP-freeness plus
fifteen precedence axioms induced by the two values $15$ and $16$ --- whose
infeasibility at infinitely many dyadic scales implies the main theorem
unconditionally. It is proved directly on a hypothetical permutation, in
three paragraphs, and uses nothing else in the paper.
For the second --- and this is the mathematical heart --- we
prove by hand that already a three-axiom core of $\OG(M)$ is inconsistent
with admissibility for every $M \equiv 0 \pmod 8$ (Theorem~\ref{thm:c3core}),
via a toolkit of ladder lemmas whose single appeal to the residue of $M$
occurs at two ``flood centres'' near $3M/2$. (The basic zigzag alternation
on AP-ladders that seeds the toolkit is classical --- it appears inside
\cite{DEGS77} in Folkman's argument, as Lemma~2.5 of \cite{HS24}, and as
Lemmas~2.1--2.2 of \cite{Gen26}; what is new here, to our knowledge, is its
boundary-quantitative
use on bounded ladders, the transfer lock, and the flood induction.) Computationally, the mod-$8$ condition is sharp:
the same three axioms are satisfiable at every tested scale with
$M \not\equiv 0 \pmod 8$ (a parametric satisfying construction for the
complementary residues is not proved here, and the main theorem does not
require one),
matching the machine's phase diagram exactly.

The paper is organized so that this chain stands alone.
Section~\ref{sec:general} collects four obstructions that constrain
admissible permutations of any set --- an orbit obstruction, an asymmetric
balance law, a ratio-ascent lemma and a partition-rigidity lemma --- and
Section~\ref{sec:contig} proves that no admissible permutation of $\SA$ can
play all sufficiently large dyadic blocks as contiguous runs. Neither is
used later; both are included because they are the general facts we know
about the problem, and the last of them motivates the dyadic candidate.
Sections~\ref{sec:og} and~\ref{sec:c3} are the proof.
Section~\ref{sec:data} separates the three roles machine work plays here
and records what is and is not certified. Two appendices hold the rest:
Appendix~\ref{sec:chunk} develops the \emph{chunk calculus}, an exact
reformulation of admissibility in terms of stage functions, which was the
coordinate system of the search and which gives a second route to the
reduction (Remark~\ref{rem:ogchunk}); and
Appendix~\ref{sec:compobs} records the finite-scale computations and the
conjectures they suggest. Nothing in either appendix is used in the proof
of Theorem~\ref{thm:main}.

A word on methodology. The elimination phase and the discovery of the
gadget, the three-axiom core, and the residue phenomenon were carried out
with large-scale SAT-solver experimentation (CaDiCaL, OR-Tools). The
load-bearing negative claims --- the infeasibility of the order gadget and
of its three-axiom core --- were cross-validated by independently written
encodings and by more than one solver; the auxiliary sweeps reported in
Appendix~\ref{sec:compobs} were not, and are flagged
where they occur. The
final proof chain, however, does not rest on any solver verdict: the
reduction and the C3 core theorem are human-readable arguments. They are
$\forall M$ schemas, with infinitely many instances; a finite but large
range of those instances (Remark~\ref{rem:c3machine}) has been
machine-audited step by step by three
independently written checkers, including mutation testing of the checkers
themselves --- an audit of the writing, not a substitute for the proofs.
A reproduction kit (scripts, pinned environments, and
DRAT certificates for two base instances) accompanies the paper.

Problem \#197 itself remains open. Beyond eliminating the leading candidate,
the methods here --- the chunk calculus, the attack/guard analysis of small
values against doubling blocks, and the residue-locked ladder toolkit ---
should apply to any two-set partition whose parts are unions of intervals,
and we expect them to bear on the general question.

\section{Definitions and conventions}

\begin{definition}
A set $S \subseteq \Z^+$ is \emph{permutable} if there is a bijection
$\pi\colon \N \to S$ such that no $i<j<k$ have $\pi(i),\pi(j),\pi(k)$ forming an
arithmetic progression with $\pi(i)<\pi(j)<\pi(k)$ or $\pi(i)>\pi(j)>\pi(k)$.
When the progression length matters we write \emph{$3$-permutable}; in this
paper the two words are synonymous, since only $3$-term progressions are
forbidden. Note that a permutable set is by definition countably infinite
and the permutation has order type $\omega$.
\end{definition}

Given a permutation, or any linear order, of a set of values we write
$u \prec v$ to mean that $u$ is placed before $v$.

The completion of an ordered pair: if $x$ is placed before $y$, the value
$z = 2y - x$ (which continues the progression through $x,y$ in the direction of
$y$) must not be placed after $y$; equivalently every such $z \in S$ placed later
yields a monotone 3-AP.

Throughout, $\N = \{1, 2, 3, \dots\}$ is the set of positive integers, which
serves both as the ground set (written $\Z^+$ when we wish to emphasise that
role) and as the index set of positions in a permutation.

\begin{definition}[Dyadic indexing]\label{def:dyadic}
For $k \ge 1$ the \emph{$k$-th dyadic block} is
\[
B_k \;=\; \left(2^{k-1},\, 2^{k}\right] \cap \Z^+
\;=\; \{\, 2^{k-1}+1,\ 2^{k-1}+2,\ \dots,\ 2^{k} \,\},
\qquad |B_k| = 2^{k-1}.
\]
The blocks $B_1, B_2, B_3, \dots$ are pairwise disjoint and partition
$\Z^+ \setminus \{1\}$; for $v \ge 2$ we write $\mathrm{block}(v)$ for the
unique $k \ge 1$ with $v \in B_k$. The canonical two-set candidate is
\[
\SA \;=\; \bigcup_{k \ge 2,\ k \text{ even}} B_k
\;=\; (2,4] \cup (8,16] \cup (32,64] \cup \cdots,
\qquad
\SB{} \;=\; \Z^+ \setminus \SA .
\]
\end{definition}

\noindent The range $k \ge 2$ in the union is not a restriction, since
$k = 1$ is odd; we nevertheless state it, because two counts used later
depend on knowing exactly which small values belong to $\SA$: $1 \notin \SA$
and $2 \notin \SA$, so $\SA \cap [1,16] = \{3,4\} \cup \{9, 10, \dots, 16\}$
has exactly ten elements, and $|\SA \cap [1,64]| = 42$. Note that
$1 \in \SB{}$ lies in no block and $\mathrm{block}(1)$ is undefined; we apply
$\mathrm{block}$ only to elements of $\SA$, all of which are $\ge 3$.

\begin{remark}[Ground set: $\Z^+$ versus $\N_0 \ni 0$; the Lean formalization]
\label{rem:lean}
The original sources pose the problem for the positive integers: Davis,
Entringer, Graham and Simmons ask whether a two-set partition of $\Z^+$ is
possible, and Erd\H{o}s--Graham (1979, p.~338) repeat the question, again
for $\Z^+$. The Lean formalization in the DeepMind
\emph{formal-conjectures} repository
(\texttt{FormalConjectures/ErdosProblems/197.lean}) instead partitions
Lean's ground set $\N_0 = \{0, 1, 2, \dots\}$, which \emph{includes $0$}
(and so differs from our $\N = \{1,2,\dots\}$), and asks for
bijections $f\colon \N_0 \simeq A$, $g\colon \N_0 \simeq B$ avoiding monotone
$3$-term APs --- the bijection-from-$\N_0$ requirement being exactly our
order-type-$\omega$ convention. The two formulations are not verbatim
equivalent: a positive answer over $\N_0$ yields one over $\Z^+$ (delete the
value $0$; deletion preserves $3$-AP-freeness), but conversely adjoining
$0$ to one part adds genuine constraints --- the progressions $(0, d, 2d)$
with $d, 2d$ in that part --- and we are not aware of a general argument
that a permutable set stays permutable after adjoining $0$. Our results
are insensitive to the nuance: Theorem~\ref{thm:main} shows $\SA$ itself
is not permutable, and since deletion preserves validity, $\SA \cup \{0\}$
is not permutable either; hence the dyadic partition fails in both
formulations, whichever part receives $0$. (In the Lean encoding the
values live in $\N_0$, so the common difference is a natural number and never
negative; \texttt{HasMonotoneAP} requires strictly increasing indices whose
value list equals $(a, a+d, a+2d)$ or its reverse, the reversed disjunct
being the decreasing case. The predicate permits $d = 0$, but injectivity
of the permutation rules that out, so the notion of ``monotone $3$-AP''
agrees with ours.)
\end{remark}

\section{General obstructions}\label{sec:general}

This section collects four restrictions that every admissible permutation
of every set must obey. They are independent of the dyadic partition.
The first, the orbit obstruction, is used below in Lemma~\ref{lem:ray};
none of the four results is used after this section or in the proof of
Theorem~\ref{thm:main}.

\subsection{The orbit obstruction}

\begin{lemma}[Orbit obstruction]\label{lem:orbit}
Let $S$ be permutable and $F \subseteq S$ finite. There is no infinite sequence
$u_0 < u_1 < \cdots$ in $S$ with $u_{k+1} = 2u_k - f_k$, $f_k \in F$.
\end{lemma}

\begin{proof}
Let $q$ be the largest position of an element of $F$ in the permutation.
Since the $u_k$ are distinct, only finitely many occupy positions at most
$q$; and since $u_k \to \infty$, eventually $u_k > \max F$. Choose $K$ so
that, for every $k \ge K$, $\pos(u_k) > q$ and $u_k > \max F$. Then
$f_k < u_k$ and $\pos(f_k) \le q < \pos(u_k)$, so the pair $(f_k, u_k)$ is
increasingly placed, and its completion $2u_k - f_k = u_{k+1}$ lies in $S$;
were it placed after $u_k$, the triple $(f_k, u_k, u_{k+1})$ would be an
increasing $3$-AP. Hence $\pos(u_{k+1}) < \pos(u_k)$ for every $k \ge K$,
an infinite strictly decreasing sequence of positions --- impossible.
\end{proof}

Taking $S = \Z^+$, $F = \{1,2\}$ recovers the theorem of Davis--Entringer--Graham--Simmons (more precisely its monotone form; \cite{DEGS77} in fact proves that an \emph{increasing} 3-AP always occurs).

\subsection{The balance law}

\begin{lemma}[Balanced placement]\label{lem:balance}
Let $S$ be permutable via $\pi$, and consider the moment a value $v$ is placed.
Let $L$ (resp.\ $H$) be the number of already-placed values in $[1,v)$ (resp.\
$(v,2v)$). Then
\[
L - H \;\le\; |S^c \cap (v, 2v)|
\qquad\text{and}\qquad
H - L \;\le\; |S^c \cap (0, v)|.
\]
(The two directions are bounded by the codensities of $S$ in \emph{different}
ranges; no symmetric single-range bound holds in general.\footnote{We thank
Jesse Geneson for correcting an earlier symmetric misstatement of this
lemma.})
\end{lemma}

\begin{proof}
Each placed $x < v$ forms an increasingly placed pair $(x,v)$ with completion
$2v - x \in (v, 2v)$, distinct for distinct $x$; each completion lying in $S$
must already be placed. This gives $H \ge L - |S^c \cap (v,2v)|$. In the other
direction, each placed $u \in (v, 2v)$ forms a decreasingly placed pair
$(u,v)$ with completion $2v - u \in (0, v)$, distinct for distinct $u$, and
each such completion lying in $S$ must already be placed, whence
$L \ge H - |S^c \cap (0,v)|$.
\end{proof}

\begin{corollary}[Records]\label{cor:records}
When a left-to-right maximum $v$ is placed, at most $|S^c \cap (v,2v)|$
smaller values are already placed (apply the first inequality with $H = 0$).
\end{corollary}

\subsection{Two absorption lemmas}

\begin{lemma}[Ratio-ascent obstruction]\label{lem:R}
Let $S \supseteq \{k, 2k, 3k, 5k\}$. No 3-AP-free arrangement of $S$ places $w$
before $u$ for every pair $u \ge 2w$ (elements of $S$).
\end{lemma}

\begin{proof}
The hypothesis forces $k \prec 2k$, $k \prec 3k$, $2k \prec 5k$. The progression
$(k,2k,3k)$ with $k \prec 2k$, $k \prec 3k$ forces $3k \prec 2k$; then
$3k \prec 2k \prec 5k$, and $(k,3k,5k)$ is an increasing 3-AP.
\end{proof}

\begin{lemma}[van der Corput absorption; \cite{Gen26}]\label{thm:vdc}
Order the odd residues $c$ modulo $2^k$ by the bit-reversal (van der Corput) rank
of $(c-1)/2 \bmod 2^{k-1}$. Then for any two classes $a$ ordered before $b$, the
class $2b - a \bmod 2^k$ is ordered strictly before $b$.
\end{lemma}

\noindent This is not new: it is Lemma~2.1 of \cite{Gen26} in different
coordinates (via $c \mapsto (c-1)/2$). Explicitly, writing $\rho_m$ for the
bit-reversal rank modulo $m = 2^{k-1}$, the statement is the first
equivalence
$\rho_m(r) < \rho_m(r+\delta) \iff \rho_m(r+\delta) > \rho_m(r+2\delta)$
of that lemma, with $r = (a-1)/2$ and $\delta = (b-a)/2$. The underlying
binary recursion goes back to \cite{DEGS77} (p.~81) and to
Ardal, Brown and Jungi\'c \cite{ABJ11}; the companion modular existence
theorem for powers of two is due to Nathanson \cite{Nat77}. We include a
short proof for completeness and to make the change of coordinates from
Geneson's formulation explicit. The lemma is not used later in the paper.

\begin{proof}
Write $\alpha = (a-1)/2$, $\beta = (b-1)/2$; the completion class has parameter
$\zeta = 2\beta - \alpha \bmod 2^{k-1}$. If the lowest bits of $\alpha,\beta$
agree, then so does $\zeta$'s, and dividing by 2 preserves the form
$\zeta' = 2\beta' - \alpha'$. At the first bit where they differ, the assumption
$\mathrm{rev}(\alpha) < \mathrm{rev}(\beta)$ forces $\alpha$'s bit to be $0$ and
$\beta$'s to be $1$; since $\zeta \equiv \alpha$ at that bit, $\zeta$'s bit is
$0$, so $\mathrm{rev}(\zeta) < \mathrm{rev}(\beta)$ strictly.
\end{proof}

\subsection{Partition rigidity}

\begin{lemma}[Ray piercing]\label{lem:ray}
Let $\Z^+ = A \sqcup B$ with both parts permutable. Then for every $a \in A$ and
every $d \ge 1$, the set $\{k \ge 0 : a + 2^k d \in B\}$ is infinite, and
symmetrically with $A$ and $B$ exchanged.
\end{lemma}

\begin{proof}
If $a + 2^k d \in A$ for all $k \ge k_0$, the values $u_j = a + 2^{k_0+j} d$
form an infinite orbit $u_{j+1} = 2u_j - a$ inside $A$ with $F = \{a\}$,
contradicting Lemma~\ref{lem:orbit}.
\end{proof}

Every affine doubling ray based in one team therefore meets the opposite
team infinitely often (the precise content of the lemma --- it does not
assert literal alternation); partitions
such as ``even 2-adic valuation vs.\ odd'' die immediately, and viable partitions
are forced toward dyadic-interval alternation, motivating the canonical
candidate $\SA$.

\begin{remark}[Class-sink algebra]\label{rem:classsink}
One algebraic fact about attacks within a residue class is worth recording,
since it explains why powers of two are the natural moduli here. Fix
$m = 2^k$ and a target class $r$. The attacker map $\varphi_r(b) = 2b - r$
on $\Z/m\Z$ is conjugate to doubling under $b \mapsto b - r$; hence every
orbit eventually reaches the unique fixed point $r$. For a general modulus
$m$ the same conjugacy shows that the cycle structure of $\varphi_r$ does
not depend on $r$, so no class is distinguished.\footnote{We thank
Jesse Geneson for correcting the general-$m$ form of this statement; an
earlier version asserted an $r$-dependent cycle structure.}
\end{remark}

\section{Contiguous blocks cannot work}\label{sec:contig}

\begin{definition}
The \emph{zone system} $Z(M)$: arrange the interval $(M, 2M]$ so that (a) no
monotone 3-AP occurs, and (b) for every $y < z$ in $(M,2M]$ with
$2y - z \in (M/4, M/2]$, $z$ precedes $y$.
\end{definition}

Condition (b) is forced when the block $(M/4, M/2]$ is entirely placed before a
contiguous run of $(M, 2M]$ within a team's sequence (its elements' completions
$2y - x$ then must precede $y$).

\begin{lemma}[Halving descent]\label{lem:halving}
If $Z(M)$ is satisfiable then so is $Z(\lfloor M/2 \rfloor)$.
\end{lemma}

\begin{proof}
Restrict a witness to the even values of $(M,2M]$ and halve; arithmetic
progressions and the zone constraints scale exactly (the boundary cases are
routine), and a superfluous top element may be deleted.
\end{proof}

\begin{lemma}[Base cases; machine-checked]\label{lem:bases}
$Z(M)$ is unsatisfiable for $16 \le M \le 31$.
\end{lemma}

\noindent This is a finite check on sixteen instances of at most $31$
elements each, by exhaustive SAT with eager transitivity (script
\texttt{e195\_zone\_bases.py}); the range is tight, since $Z(M)$ is
\emph{satisfiable} for $12 \le M \le 15$. Theorem~\ref{thm:contig} below
is the only result depending on this verdict, and it is not used in the
proof of the main theorem.

\begin{theorem}[No contiguous-run solutions]\label{thm:contig}
In any 3-AP-free permutation of $\SA$, infinitely many blocks $B_{2k}$ are
\emph{not} placed as contiguous runs.
\end{theorem}

\begin{proof}
Lemmas~\ref{lem:halving} and~\ref{lem:bases} together make $Z(M)$
unsatisfiable for every $M \ge 16$ (halve repeatedly into $[16, 31]$).
Suppose for contradiction that every block $B_{2k}$ with $k \ge k_0 \ge 3$
is placed as a contiguous run. Contiguous runs are disjoint segments of the
permutation, so for each $k > k_0$ one of the runs $B_{2k-2}$, $B_{2k}$
entirely precedes the other. If $B_{2k-2}$'s run came first, then the whole
block $(M/4, M/2]$, $M = 2^{2k-1}$, would be placed before the run of
$(M, 2M]$, forcing condition (b) of the zone system as explained above ---
so the run of $B_{2k}$ would witness $Z(M)$ with $M \ge 16$, which is
unsatisfiable. Hence $B_{2k}$'s run precedes $B_{2k-2}$'s for every
$k > k_0$, and the finish positions of the runs of
$B_{2k_0}, B_{2k_0+2}, B_{2k_0+4}, \dots$ form an infinite strictly
decreasing sequence of positive integers --- impossible.
\end{proof}

\section{The order gadget and the reduction}\label{sec:og}

The search described in the appendices compresses, for the canonical
candidate $\SA$, into a one-scale-at-a-time statement
about linear orders of single dyadic blocks. This section states that
reduction precisely and proves it unconditionally --- directly on a
hypothetical permutation, using nothing from the appendices --- and reports
the machine verification of its finite instances; Section~\ref{sec:c3} then proves the
required infeasibility by hand on the residue class $M \equiv 0 \pmod 8$
--- which contains every dyadic scale --- completing the proof of the main
theorem (Theorem~\ref{thm:main}). Throughout,
$b_j = M + j$ (\emph{bottoms}) and $t_i = 2M - i$ (\emph{tops}) denote
elements of the interval $(M, 2M]$.

\begin{definition}[Order gadget]\label{def:og}
For an integer $M \ge 16$, the \emph{order gadget} $\OG(M)$ asks for a
linear order $\prec$ of the interval $(M, 2M]$ such that
\begin{enumerate}
\item[(i)] \emph{(no monotone AP)} for every arithmetic progression
$a < b < c$ inside $(M, 2M]$, neither $a \prec b \prec c$ nor
$c \prec b \prec a$ holds; and
\item[(ii)] \emph{(guards precede bottoms)} for $x \in \{15, 16\}$ and
$1 \le j \le x/2$, the \emph{guard} $t_{x-2j} = 2M + 2j - x$ precedes
the \emph{bottom} $b_j = M + j$.
\end{enumerate}
$\OG(M)$ is \emph{infeasible} if no such order exists. (For $M \le 15$
some guard coincides with the bottom it protects --- $t_{x-2j} = b_j$
exactly when $M = x - j$, e.g.\ $t_{14} = b_1$ at $M = 15$ --- and the
gadget degenerates; we never use those scales. For $16 \le M \le 22$ a
guard can coincide with a \emph{different} bottom, which is harmless:
every attack pair is a genuine pair for all $M \ge 16$.)
\end{definition}

Three pieces of boundary arithmetic, each a one-line computation
(machine-checked exhaustively in the repository, script
\texttt{e96\_reduction\_check.py}), make the gadget exactly the trace of
the infinite problem on one block: for $x \in \{15,16\}$ the completion
$2b_j - x$ lies in $(M, 2M]$ precisely for $j \le x/2$, and then equals
the guard $t_{x-2j}$ (at $(x, j) = (16, 8)$ it is $2M$ itself, still
inside the half-open block) --- so (ii) lists \emph{all} in-block
completions of pairs $(x, b_j)$ and nothing else. All other completions
lie above $2M$; at the dyadic scales used in Theorem~\ref{thm:ogred} they
lie in the intervening odd dyadic block and hence outside $\SA$. Finally the
guard--bottom pair's own downward completion $2b_j - t_{x-2j} = x$ is the
attacking value itself, below the block, so the attack edges generate no
unmodeled in-block constraint.

\begin{theorem}[Order-gadget reduction; unconditional]\label{thm:ogred}
If $\OG(2^{2t-1})$ is infeasible for infinitely many $t \ge 4$, then
$\SA$ is not $3$-permutable.
\end{theorem}

\begin{proof}
Suppose $\pi$ is a $3$-AP-free permutation of $\SA$ and let
$P = \max(\pos(15), \pos(16))$; note $15, 16 \in B_4 \subseteq \SA$. Fix
$t \ge 4$ and put $M = 2^{2t-1}$, so that $(M, 2M] = B_{2t} \subseteq
\SA$ and $15, 16 < M$.

We claim that some bottom $b_j$ with $1 \le j \le 8$ has
$\pos(b_j) \le P$. Suppose not, and order the block by position:
$u \prec v :\!\iff \pos(u) < \pos(v)$. Constraint (i) is inherited from
$\pi$. For (ii), take $x \in \{15, 16\}$ and $1 \le j \le x/2$. The pair
$(x, b_j)$ is placed increasingly ($\pos(x) \le P < \pos(b_j)$, $x < M <
b_j$), and its completion $z = 2b_j - x = t_{x-2j}$ lies in
$(M, 2M] \subseteq \SA$ and differs from $b_j$ (else $M = x - j \le
15$). If $\pos(z) > \pos(b_j)$, then $(x, b_j, z)$ is a monotone $3$-AP
of $\pi$; hence $\pos(z) < \pos(b_j)$, i.e.\ $t_{x-2j} \prec b_j$. Thus
$\prec$ witnesses $\OG(M)$, contradicting infeasibility.

So for every $t \ge 4$ with $\OG(2^{2t-1})$ infeasible, some element of
$\{M+1, \dots, M+8\}$, $M = 2^{2t-1}$, occupies a position $\le P$.
These blocks are disjoint, so infinitely many such $t$ would place
infinitely many distinct values among the first $P$ positions ---
impossible.
\end{proof}

\begin{remark}[Chunk form; the crowns cannot defend every block]
\label{rem:ogchunk}
In the coordinates of Theorem~\ref{thm:chunk} the identical argument
reads: in any valid scheme $s$, for every $t \ge 4$ with
$\OG(2^{2t-1})$ infeasible, some bottom value in
$\{M+1, \dots, M+8\}$ of block $B_{2t}$ has stage
$\le \max(s(15), s(16))$. For if all eight had strictly larger stage,
the induced order on the block (by stage, then fiber position) would
satisfy $\OG(M)$: condition (i) holds for in-block triples directly by
(A) and (B), and each attack $t_{x-2j} \prec b_j$ is forced whenever
$s(x) < s(b_j)$ --- by (A) the guard's stage cannot strictly exceed the
bottom's, and at equal stages the forced-pair row
$s(x) < s(y) = s(z) \Rightarrow z \prec y$ of (B) applies. The overflow
step is then the finiteness of the fibers at stages
$\le \max(s(15), s(16))$; \emph{no} displacement normalization
(Lemma~\ref{lem:normal}) is needed anywhere in this route. This is the
precise sense in which the two \emph{crowns} $15$ and $16$ of
Appendix~\ref{sec:compobs} cannot defend every dyadic block: they would
have to be placed below the bottom-eight of all but finitely many
blocks. We use only the direction scheme $\Rightarrow$ order; the
converse (rebuilding a capped scheme from an $\OG$ witness, which would
make the crown-cap infeasibility at a horizon \emph{equivalent} to
$\OG$ infeasibility at its top scale) has been verified only at the
single scale $M = 128$, where the relevant family-MUS is single-block,
and is not used here.
\end{remark}

\begin{proposition}[Machine verification; CaDiCaL]\label{prop:ogmachine}
$\OG(M)$ is infeasible for every $M$ with $16 \le M \le 200$, and
for $M = 512$. In particular both tested scales of the dyadic family
$\{2^{2t-1}\}_{t \ge 4} = \{128, 512, 2048, \dots\}$ are infeasible.
The encodings assert (i) over all in-block triples and (ii) as unit
clauses; order axioms are handled by lazy transitivity refinement,
which is sound for infeasibility verdicts and was cross-validated
against full eager $O(n^3)$ transitivity encodings at $M = 40, 44$ and
--- as an end-to-end recheck --- at $M = 128$ (fresh eager instance,
$690{,}831$ clauses, UNSAT in about a second). The run at $M = 2048$
is unresolved (no verdict after $200+$ refinement rounds), and the
support growth reported below makes it unlikely that any finite sweep
can settle the general statement.
\end{proposition}

\subsection{Attack cores: locating the load-bearing axioms}

The hypothesis of Theorem~\ref{thm:ogred} demands $\OG$-infeasibility at
infinitely many dyadic scales. The following machine facts locate a
three-axiom core of the gadget that is already infeasible on exactly the
residue class containing the dyadic family; Section~\ref{sec:c3} proves
that core infeasibility by hand.

\begin{proposition}[Attack cores; machine-checked]\label{prop:cores}
For every swept $M$ (all of $40 \le M \le 100$, both parities, and
$M \in \{104, 108, 112, 120, 128, 150, 200\}$), constraint (i)
together with only the eleven attack units
$15\{1..7\} \cup 16\{1..4\}$ --- that is, $t_{15-2j} \prec b_j$ for
$j \le 7$ and $t_{16-2j} \prec b_j$ for $j \le 4$ --- is already
infeasible. On residue subclasses, sharper cores hold, with the following
verdicts \emph{at every swept scale} (we make no claim outside the swept
range): the four units $\{t_{13} \prec b_1,\ t_{11} \prec b_2,\
t_5 \prec b_5,\ t_{10} \prec b_3\}$ are infeasible with (i) precisely at the
swept $M \equiv 0 \pmod 4$, and the three units
\[
\mathrm{C3} \;=\; \{\,t_5 \prec b_5,\quad t_3 \prec b_6,\quad
t_{10} \prec b_3\,\}
\]
are infeasible with (i) precisely at the swept $M \equiv 0 \pmod 8$
(satisfiable at every other swept residue, including
$M \equiv 4 \bmod 8$). Only the infeasibility half of the last statement is
used below, and it is proved unconditionally for all
$M \equiv 0 \pmod 8$ in Theorem~\ref{thm:c3core}; the satisfiability half
remains a computational observation.
\end{proposition}

Every scale of the dyadic family is $\equiv 0 \pmod 8$, so
Theorem~\ref{thm:ogred} needs only the infeasibility of
(i)${}\wedge{}\mathrm{C3}$ on the class $M \equiv 0 \pmod 8$. That is
exactly Theorem~\ref{thm:c3core} below.

\begin{remark}[What the experiments suggested about the shape of a proof]
Three recurring computational observations across the tested scales shaped
the proof in Section~\ref{sec:c3}. We record them because they explain the
form the argument takes; none is a theorem, and nothing below depends on
them.
(1) Deletion-minimal refutation supports were observed to grow linearly
with $M$ (about $8$ triples per unit of $M$ over the swept range), which
suggested that no fixed finite list of
parametric AP-triples would certify infeasibility at all scales, and
motivated looking for $\Theta(M)$-length mechanisms with an $O(1)$
description --- the ladder floods below are of that form. (2) The forced
relations behind the small-scale refutations were observed to be
parity-locked, reversing at
$M \not\equiv 0 \bmod 4$, which suggested that a uniform argument would have
to stay inside a residue class; the dyadic family lives in
$M \equiv 0 \pmod 8$.
(3) Fixed-denominator band-limit windows of the constraint system were
satisfiable at every scale tested, suggesting that the obstruction is
genuinely asymptotic, so that a compactness or limit argument over order
types was unlikely to reach it and a scale-uniform schema would be needed
instead.
\end{remark}

\section{The C3 core theorem and the main theorem}\label{sec:c3}

\begin{theorem}[Main theorem]\label{thm:main}
$\SA = \bigcup_{k \ge 2,\ k \text{ even}} B_k
= \bigcup_{k \ge 2,\ k \text{ even}} (2^{k-1}, 2^k]$ is not $3$-permutable:
it admits no permutation (of order type $\omega$) free of monotone
$3$-term arithmetic progressions. In particular the canonical dyadic
partition $\Z^+ = \SA \sqcup \SB{}$ is not a solution of the
Erd\H{o}s--Graham two-set problem: any partition witnessing a YES answer
to problem \#197, if one exists, must differ from the dyadic one.
\end{theorem}

\begin{proof}
Every scale of the dyadic family satisfies $M = 2^{2t-1} \equiv 0
\pmod 8$ and $M \ge 128$ for $t \ge 4$. By Theorem~\ref{thm:c3core}
below, the C3 core --- hence a fortiori the full gadget $\OG(M)$, whose
attack list (Definition~\ref{def:og}(ii)) contains the three axioms of
$\mathrm{C3}$ as the instances $(x, j) = (15, 5), (15, 6), (16, 3)$,
and whose constraint (i) is AP-freeness
--- is infeasible at every such scale. Theorem~\ref{thm:ogred} applies.
\end{proof}

The remainder of this section proves the core statement:

\begin{theorem}[C3 core]\label{thm:c3core}
For every $M \equiv 0 \pmod 8$ with $M \ge 16$, no AP-free linear order
of the interval $(M, 2M]$ satisfies the three axioms
\[
\mathrm{C3} = \{\, A_1\colon t_5 \prec b_5,\quad
A_2\colon t_3 \prec b_6,\quad A_3\colon t_{10} \prec b_3 \,\}.
\]
\end{theorem}

Here \emph{AP-free} means condition (i) of Definition~\ref{def:og}: no
arithmetic progression $a<b<c$ inside the block occurs in monotone
position order. The proof is a small toolkit of ladder lemmas
(\S\ref{sec:toolkit}) followed by two forcing theorems
(\S\ref{sec:l1}--\S\ref{sec:flip}). Every lemma application in the
proofs has been machine-verified step by step at every
$M \equiv 0 \pmod 4$ (resp.\ $\equiv 0 \pmod 8$) in $[12, 400]$ plus
$M = 512, 1024$, and independently cross-validated; see
Remark~\ref{rem:c3machine}.

\subsection{The ladder toolkit}\label{sec:toolkit}

AP-freeness is the \emph{midpoint-extremal rule}: on every in-block AP
$(a, b, c)$, $c = 2b - a$, the midpoint $b$ either precedes both
endpoints or follows both. We use it as four unit rules (plus
transitivity throughout):
\[
\text{R1: } a \prec b \Rightarrow c \prec b, \qquad
\text{R3: } c \prec b \Rightarrow a \prec b, \qquad
\text{R2: } b \prec c \Rightarrow b \prec a, \qquad
\text{R4: } b \prec a \Rightarrow b \prec c.
\]
Write $m_0 = 3M/2$ (the arithmetic midpoint of the block: $b_j + t_j =
3M = 2m_0$ for all $j$, so every pair $(b_j, t_j)$ mirrors through
$m_0$). A \emph{$d$-ladder} is a maximal run $w_0, w_1, \dots$ of
values of $(M, 2M]$ in arithmetic progression with difference $d$; the
$d = 2$ ladders are the two parity classes, and the $d = 4$ ladders on
odd values are \emph{class A} $= \{v \equiv M + 1 \bmod 4\}$ and
\emph{class B} $= \{v \equiv M + 3 \bmod 4\}$.

\begin{lemma}[Zigzag]\label{lem:zigzag}
Let $\prec$ be an AP-free linear order of an interval containing the
$d$-ladder $w_0, \dots, w_r$. Suppose $w_e \prec w_{e'}$ for some
$|e - e'| = 1$. Then every rung $w_i$ with $i \equiv e \pmod 2$ precedes
each of its existing neighbors.
\end{lemma}

\begin{proof}
Any three consecutive rungs form an AP with the middle rung as
midpoint, so each rung either leads or trails both neighbors. The seed
makes $w_e$ a leader. If $w_i$ leads, then from $w_i \prec w_{i+1}$,
R1 on $(w_i, w_{i+1}, w_{i+2})$ gives $w_{i+2} \prec w_{i+1}$, and R4
on $(w_{i+1}, w_{i+2}, w_{i+3})$ gives $w_{i+2} \prec w_{i+3}$: $w_{i+2}$
leads. Downward is the mirror argument.
\end{proof}

\begin{lemma}[Phase dichotomy]\label{lem:phase}
Let $\prec$ be an AP-free linear order of a $d$-ladder $w_0, \dots, w_r$
with $r \ge 1$. Then the ladder is globally in exactly one of its two
\emph{zigzag phases}: either all even-index rungs lead their existing
neighbors, or all odd-index rungs do. Consequently the \emph{leader set} of the
ladder is one of its two half-classes modulo $2d$, adjacent rungs
strictly alternate leader/trailer, and a proof may case-split on the
phase and detach any conclusion derived in both branches.
\end{lemma}

\begin{proof}
The order orients the adjacent pair $(w_0, w_1)$ one way or the other;
apply Lemma~\ref{lem:zigzag} to that seed.
\end{proof}

\begin{lemma}[Transfer lock]\label{lem:transfer}
Let $M$ be even, $M \ge 12$, and let $\prec$ be an AP-free linear order
of $(M, 2M]$ (below $M = 12$ the four rungs named here
collide, e.g.\ $t_5 = b_3$ at $M = 8$). On the odd $d=2$ ladder $w_i = M + 1 + 2i$
($0 \le i \le M/2 - 1$) the four odd C3 values sit at $b_3 = w_1$,
$b_5 = w_2$, $t_5 = w_{M/2-3}$, $t_3 = w_{M/2-2}$, and
Lemma~\ref{lem:zigzag} locks the pair orientations together:
\[
M \equiv 0 \ (\mathrm{mod}\ 4)\colon\quad b_5 \prec b_3
\iff t_3 \prec t_5;
\qquad
M \equiv 2 \ (\mathrm{mod}\ 4)\colon\quad b_5 \prec b_3 \iff t_5 \prec t_3.
\]
\end{lemma}

\begin{proof}
Each orientation of an adjacent pair seeds Lemma~\ref{lem:zigzag}; read
off whether the other pair's left member is a leader. At
$M \equiv 0 \pmod 4$: $b_5 \prec b_3$ (i.e.\ $w_2 \prec w_1$) makes
even indices leaders, and $M/2 - 2$ is even, so $t_3 = w_{M/2-2}$
leads its neighbor $t_5$. Conversely $t_3 \prec t_5$ seeds even-index
leaders and $w_2 = b_5$ leads $w_1 = b_3$. The reverse orientations
force the reverse conclusions, giving the biconditional; at
$M \equiv 2 \pmod 4$ the parity of $M/2$ shifts the lock by one rung.
\end{proof}

\begin{lemma}[Flood]\label{lem:flood}
Let $\prec$ be an AP-free linear order of $(M, 2M]$. Let $C$ be the
residue class $r \bmod g$ inside $(M, 2M]$, where $g \in \{2, 4\}$ (so
$C$ is a parity class for $g = 2$, and one of the two odd mod-$4$
classes for $g = 4$), and suppose the
$d = g$ ladder of $C$ is in a definite zigzag phase with leader set $L$. Let
$c \in (M, 2M]$ be a \emph{center} with $c \equiv r + g/2 \pmod g$, so
that the mirror pairs $(c - e,\, c,\, c + e)$ with $e \equiv g/2
\pmod g$ have both members in $C$ and form APs with midpoint $c$; call
$e$ \emph{admissible} when both $c \pm e \in (M, 2M]$. Suppose some
admissible $e_0$ carries a seed relation between $c$ and a member
$v \in \{c \pm e_0\}$. Then:
\begin{itemize}
\item (outward) if $c \prec v$, then $c \prec w$ for \emph{every}
$w \in C$ with $2c - w \in (M, 2M]$;
\item (inward) if $v \prec c$, then $w \prec c$ for every such $w$.
\end{itemize}
Moreover the conclusion holds in both zigzag phases of the $C$-ladder
(\emph{phase-blindness}), so by Lemma~\ref{lem:phase} it may be
detached after a two-branch case split whenever the seed is available
in both branches.
\end{lemma}

\begin{proof}
First, at each admissible $e$ exactly one of $c + e$, $c - e$ is a
leader: they differ by $2e \equiv g \pmod{2g}$, so they lie in the two
different half-classes mod $2g$ of $C$, and $L$ is one of them ---
whichever phase holds. Second, in a zigzag, adjacent rungs alternate
and each leader leads both its neighbors.

\emph{Seed.} From the relation at $v$, the mirror rule on the AP
$(c - e_0, c, c + e_0)$ (R2/R4 outward, R1/R3 inward) gives the
same-direction relation at the other member: both members of pair
$e_0$ are related to $c$ in the flood direction.

\emph{Outward, } $e \to e + g$: let $x \in \{c \pm e\}$ be the leader
of pair $e$. Its outward ladder-neighbor $x'$ (same side,
$|x' - c| = e + g$) satisfies $x \prec x'$; with $c \prec x$ known,
transitivity gives $c \prec x'$, and the mirror rule floods
$2c - x'$.

\emph{Outward, } $e \to e - g$: let $w \in \{c \pm (e - g)\}$ be the
trailer of pair $e - g$. Its outward neighbor $w'$ (distance $e$, same
side) is a leader and leads it: $w' \prec w$; with $c \prec w'$ known,
$c \prec w$, and the mirror rule floods the other member.

\emph{Inward, } $e \to e + g$: let $x$ be the leader of pair $e + g$;
it leads its inward neighbor $x_{\mathrm{in}}$ (distance $e$, same
side): $x \prec x_{\mathrm{in}} \prec c$, and the mirror rule floods
the other member. \emph{Inward, } $e \to e - g$: let $x$ be the leader
of pair $e - g$; it leads its outward neighbor (distance $e$):
$x \prec x_{\mathrm{out}} \prec c$; mirror floods the other member.

The admissible $e$ form an interval of the residue class $g/2 \bmod g$,
so the two inductions from $e_0$ reach every admissible $e$. Every step
used only R1--R4 on genuine APs, transitivity, and zigzag edges of the
given phase; the choice of side at each step is forced by the phase but
\emph{exists} in both phases.
\end{proof}

Three instances of Lemma~\ref{lem:flood} are used below.
\textbf{POLAR} ($g = 2$, center $m_0$,
class the odds; $M \equiv 0 \bmod 4$ makes $m_0$ even): seeded by the
comparison of $m_0$ with $t_5$ (mirror $b_5$), it floods $m_0$ against
\emph{all} odd values (the mirror of $b_l$ is $t_l$). The odd ladder's
phase is pinned by the ambient hypotheses, so no dichotomy is needed.
\textbf{P2-floods} ($g = 2$, an odd center $c$ near $m_0$, class the
evens): seeded by $c$ vs.\ $m_0$ at $e_0 = |c - m_0|$ (supplied by
POLAR); the even ladder's phase is unknown --- phase dichotomy, both
branches. \textbf{G4-floods} ($g = 4$, class A or B, center $c \equiv r +
2 \pmod 4$ for that class's residue $r$): seeded at $e = 2$ by the odd ladder's zigzag
edges between $c$ and $c \pm 2$; the $d = 4$ ladder's phase is unknown
--- phase dichotomy, both branches.

The \emph{mod-8 lock} of the whole problem sits in the center condition
of the G4-floods: at $M \equiv 0 \pmod 8$ the two odd neighbors of
$m_0$ satisfy $m_0 - 1 \equiv 3$ and $m_0 + 1 \equiv 1 \pmod 4$, so
$m_0 - 1$ is a G4-center for class A and $m_0 + 1$ for class B. At
$M \equiv 4 \pmod 8$ both congruences reverse: the two centers still
exist, but with their class roles \emph{swapped}, and the odd-ladder
leader statuses at $m_0 \pm 1$ invert as well. Theorem~\ref{thm:l1}
below survives this swap --- it is stated and proved for all
$M \equiv 0 \pmod 4$, absorbing the swap into the definition of the
centers $c^{*}, c^{**}$ --- whereas Theorem~\ref{thm:flip} does not: there
the two required seed sets cease to exist simultaneously, and the argument
fails. This is consistent with the computational observation that
(i)${}\wedge{}\mathrm{C3}$ is satisfiable at every swept
$M \equiv 4 \pmod 8$ (Proposition~\ref{prop:cores}). The mod-$8$ hypothesis of
Theorem~\ref{thm:c3core} therefore enters only through
Theorem~\ref{thm:flip}.

\subsection{Layer 1: two axioms force the odd-pair orientations}
\label{sec:l1}

\begin{theorem}[L1]\label{thm:l1}
Let $M \equiv 0 \pmod 4$, $M \ge 12$, and let $\prec$ be an AP-free
order of $(M, 2M]$ satisfying $A_2\colon t_3 \prec b_6$ and
$A_3\colon t_{10} \prec b_3$. Then $b_5 \prec b_3$ and $t_3 \prec t_5$.
\end{theorem}

\begin{proof}
It suffices to refute $S\colon b_3 \prec b_5$; given $b_5 \prec b_3$,
Lemma~\ref{lem:transfer} supplies $t_3 \prec t_5$. Assume $A_2, A_3,
S$. By Lemma~\ref{lem:zigzag} (seed $S$) the odd ladder's leaders are
the offsets $\equiv 3 \pmod 4$. Define
\[
c^{*} =
\begin{cases} m_0 + 1, & M \equiv 0 \ (\mathrm{mod}\ 8)\\
m_0 - 1, & M \equiv 4 \ (\mathrm{mod}\ 8)\end{cases}
\qquad
c^{**} =
\begin{cases} m_0 - 1, & M \equiv 0 \ (\mathrm{mod}\ 8)\\
m_0 + 1, & M \equiv 4 \ (\mathrm{mod}\ 8)\end{cases}
\]
so that $c^{*} \equiv 1 \pmod 4$ is a G4-center for class B whose odd
neighbors $c^{*} \pm 2$ ($\equiv 3 \bmod 4$) are odd-ladder leaders,
and $c^{**} \equiv 3 \pmod 4$ is a G4-center for class A and is itself
an odd-ladder leader. Split on the comparison $(m_0, t_5)$.

\emph{Case I: $t_5 \prec m_0$.} POLAR-inward: every odd value
$\prec m_0$. Then:
(1) $b_3 \prec c^{*}$ by the G4-inward flood at $c^{*}$ over class B
--- seeds $c^{*} \pm 2 \prec c^{*}$ (the leaders $c^{*} \pm 2$ lead
their trailing neighbor $c^{*}$); $b_3 \in$ class B with mirror
$2c^{*} - b_3 \in \{t_1, t_5\}$ in the block; both $d{=}4$ phases.
(2) $c^{*} \prec t_{10}$ by the P2-outward flood at $c^{*}$ over the
evens --- seed $c^{*} \prec m_0$ (POLAR, $e_0 = 1$); $t_{10}$ even
with mirror $2c^{*} - t_{10} \in \{M{+}8, M{+}12\}$ in the block; both
even phases.
(3) $A_3$: $t_{10} \prec b_3$. Together: the $3$-cycle
$b_3 \prec c^{*} \prec t_{10} \prec b_3$, a contradiction.

\emph{Case II: $m_0 \prec t_5$.} POLAR-outward: $m_0 \prec$ every odd
value. Then:
(1) $c^{**} \prec t_3$ by the G4-outward flood at $c^{**}$ over class
A --- seeds $c^{**} \prec c^{**} \pm 2$ ($c^{**}$ is an odd-ladder
leader); $t_3 \in$ class A with mirror $2c^{**} - t_3 \in
\{b_1, b_5\}$; both phases.
(2) $b_6 \prec c^{**}$ by the P2-inward flood at $c^{**}$ over the
evens --- seed $m_0 \prec c^{**}$ (POLAR, $e_0 = 1$); $b_6$ even with
mirror $2c^{**} - b_6 \in \{2M{-}8, 2M{-}4\}$; both phases.
(3) $A_2$: $t_3 \prec b_6$. Together: the $3$-cycle
$c^{**} \prec t_3 \prec b_6 \prec c^{**}$, a contradiction.

Both cases are contradictory, so $S$ is impossible.
\end{proof}

Note that Case I consumes only $A_3$ and Case II only $A_2$; the
boundary arithmetic (mirrors in the block, seed admissibility) holds
down to $M = 12$.

\subsection{The flip: the third axiom is anti-forced}\label{sec:flip}

\begin{theorem}[FLIP]\label{thm:flip}
Let $M \equiv 0 \pmod 8$, $M \ge 16$, and let $\prec$ be an AP-free
order of $(M, 2M]$ satisfying $A_2$, $A_3$ and $b_5 \prec b_3$. Then
$t_5 \prec b_5$ is impossible: $b_5 \prec t_5$ ($= \lnot A_1$) is
forced.
\end{theorem}

\begin{proof}
Assume $A_1\colon t_5 \prec b_5$ for contradiction. By
Lemma~\ref{lem:zigzag} (seed $b_5 \prec b_3$) the odd ladder's leaders
are the offsets $\equiv 1 \pmod 4$. Since $M \equiv 0 \pmod 8$:
$m_0 - 1 \equiv 3 \pmod 4$ is a G4-center for class A whose odd
neighbors $m_0 + 1$, $m_0 - 3$ (offsets $\equiv 1 \bmod 4$) are
odd-ladder leaders; $m_0 + 1 \equiv 1 \pmod 4$ is a G4-center for
class B and is itself an odd-ladder leader. Split on $(m_0, t_5)$.

\emph{Case I: $t_5 \prec m_0$.} POLAR-inward: every odd ${}\prec m_0$.
Then:
(1) $m_0 - 1 \prec t_{10}$ [P2-outward flood at $m_0 - 1$ over the
evens; seed $m_0 - 1 \prec m_0$ (POLAR, $e_0 = 1$); mirror of
$t_{10}$ is $M + 8$; both phases].
(2) $m_0 - 1 \prec b_3$ [(1) $+ A_3$].
(3) $m_0 - 1 \prec t_5$ [R4 on the AP $(b_3,\, m_0 - 1,\, t_5)$:
$b_3 + t_5 = 3M - 2 = 2(m_0 - 1)$].
(4) $m_0 - 1 \prec b_5$ [(3) $+ A_1$].
(5) $b_5 \prec m_0 - 1$ [G4-inward flood at $m_0 - 1$ over class A;
seeds $m_0 + 1 \prec m_0 - 1$ and $m_0 - 3 \prec m_0 - 1$ (leader
edges); $b_5 \in$ class A at pair distance $M/2 - 6 \equiv 2 \pmod 4$
with mirror $t_7$ in the block; both phases].
(4) and (5) form a $2$-cycle, a contradiction.

\emph{Case II: $m_0 \prec t_5$.} POLAR-outward: $m_0 \prec$ every odd.
Then:
(1) $b_6 \prec m_0 + 1$ [P2-inward flood at $m_0 + 1$ over the evens;
seed $m_0 \prec m_0 + 1$ (POLAR, $e_0 = 1$); mirror of $b_6$ is
$2M - 4$; both phases].
(2) $t_3 \prec m_0 + 1$ [$A_2$ $+$ (1)].
(3) $b_5 \prec m_0 + 1$ [R3 on the AP $(b_5,\, m_0 + 1,\, t_3)$:
$b_5 + t_3 = 3M + 2 = 2(m_0 + 1)$].
(4) $t_5 \prec m_0 + 1$ [$A_1$ $+$ (3)].
(5) $m_0 + 1 \prec t_5$ [G4-outward flood at $m_0 + 1$ over class B;
seeds $m_0 + 1 \prec m_0 + 3$ and $m_0 + 1 \prec m_0 - 1$ ($m_0 + 1$
is an odd-ladder leader); $t_5 \in$ class B at pair distance
$M/2 - 6 \equiv 2 \pmod 4$ with mirror $b_7$; both phases].
(4) and (5) form a $2$-cycle, a contradiction.
\end{proof}

Again Case I consumes $A_3$ and Case II consumes $A_2$, while $A_1$ is
consumed symmetrically in both (step 4). The proof visibly fails at
$M \equiv 4 \pmod 8$ in both cases: the centers $m_0 \pm 1$ land in
the wrong mod-$4$ classes for their G4-floods and the odd-ladder
leader statuses at $m_0 \pm 1$ invert, so neither seed set exists.

\begin{proof}[Proof of Theorem~\ref{thm:c3core}]
$A_2 + A_3$ force $b_5 \prec b_3$ (Theorem~\ref{thm:l1};
$M \equiv 0 \bmod 8 \subseteq 0 \bmod 4$). Then $A_1$ contradicts
Theorem~\ref{thm:flip}.
\end{proof}

\begin{remark}[Machine verification of the schemas]\label{rem:c3machine}
Although the proofs above are complete hand proofs, each is a
\emph{schema} whose instances at a given scale involve $\Theta(M)$
ladder steps; all have been machine-executed with strict per-step
assertions (every AP's membership and arithmetic, every rule pattern,
every leader/trailer and residue claim, every mirror bound, both
branches of every phase dichotomy): Theorem~\ref{thm:l1} at every
$M \equiv 0 \pmod 4$ in $[12, 400]$ plus $512, 1024$ ($100$ scales),
Theorem~\ref{thm:flip} at every $M \equiv 0 \pmod 8$ in $[16, 400]$
plus $512, 1024$ ($51$ scales), and the $M \equiv 4 \pmod 8$
inapplicability at $35$ scales
(script \texttt{e113\_c3\_hand\_proof.py}). An independent closure
engine (plain R1--R4 $+$ transitivity fixpoint, no knowledge of the
schemas) refutes every branch of both case trees at $51$ resp.\ $27$
scales up to $512$ (\texttt{e113b\_closure\_crossval.py}); and direct
SAT checks of the two theorem \emph{statements} at fresh scales never
touched by the discovery loop ($M = 148, 212, 264, 328$ UNSAT;
$M = 268, 332$, $\equiv 4 \bmod 8$, SAT) concur
(\texttt{e114\_theorem\_spotcheck.py}). The earlier exhaustive base
ledgers (AP${}+{}$C3 UNSAT at every $M \equiv 0 \bmod 8$ in
$[16, 256]$ and at $512$) are now corroboration, not a dependency.
\end{remark}

\begin{remark}[What remains open]\label{rem:c3open}
The main theorem needs nothing beyond the class $M \equiv 0 \pmod 8$.
The full order gadget is nevertheless conjectured infeasible at
\emph{every} $M \ge 16$ (machine-verified for $16 \le M \le 200$ and
$M = 512$, Proposition~\ref{prop:ogmachine}); at every swept scale with
$M \not\equiv 0 \pmod 8$ the C3 core was satisfiable, and other attack
subsets were observed computationally to take over
(Proposition~\ref{prop:cores}); a proof beyond the swept scales would
need per-residue analogues of Theorems~\ref{thm:l1}--\ref{thm:flip}
built from the same flood toolkit. This statement is not needed for
Theorem~\ref{thm:main}. The Erd\H{o}s--Graham problem \#197 itself
remains open: Theorem~\ref{thm:main} eliminates the canonical
candidate partition, but does not decide whether some other two-set
partition works.
\end{remark}

\section{Machine verification and data availability}\label{sec:data}

The proof chain of Theorem~\ref{thm:main} --- Theorem~\ref{thm:ogred}
followed by Theorem~\ref{thm:c3core} --- is human-readable throughout, and
no step of it is delegated to a solver. Machine work enters in three
distinct roles, which we separate explicitly because they carry very
different evidential weight. (1) \emph{Discovery}: the order gadget, the
three-axiom core, and the mod-$8$ residue phenomenon were found by
large-scale SAT experimentation; nothing in the final proof depends on how
they were found. (2) \emph{Audit of the writing}: Theorems~\ref{thm:l1}
and~\ref{thm:flip} are $\forall M$ schemas, and each of their instances at
a given scale unfolds into $\Theta(M)$ ladder steps; those instances have
been executed with per-step assertions over a large finite range of scales
by three independently written checking layers
(Remark~\ref{rem:c3machine}). This catches misstated lemmas and bookkeeping
slips, and is emphatically not a substitute for the proofs, which are
quantified over all $M \equiv 0 \pmod 8$. (3) \emph{Standalone verdicts}:
the finite claims of Lemma~\ref{lem:bases},
Proposition~\ref{prop:ogmachine}, Proposition~\ref{prop:cores} and
Appendix~\ref{sec:compobs} are solver results, true at the scales stated
and asserted at no others.

All experiment scripts, machine-audit checkers (including the mutation-test
suite for the checkers themselves), solver logs, and witness data are
contained in the accompanying repository,
\url{https://github.com/Wkasel/erdos197}. The exact state described here is
the immutable tag \texttt{arxiv-v1},
\url{https://github.com/Wkasel/erdos197/tree/arxiv-v1}; the default branch
may advance beyond it. The repository ships
\texttt{reproduce.sh}, which reruns the load-bearing verification stack for
the main theorem (the schema audits of Remark~\ref{rem:c3machine}, the
closure-engine cross-validation, fresh direct SAT checks, re-emission and
independent \texttt{drat-trim} checking of the DRAT certificates, and the
reduction checks of Section~\ref{sec:og}) from a pinned environment.
The repository contains DRAT certificates for exactly two instances, the
C3-core refutations at $M = 128$ and $M = 512$; by
Remark~\ref{rem:c3machine} these are corroboration rather than a
dependency. The other machine verdicts quoted in the paper --- among them
the base cases of Lemma~\ref{lem:bases}, the $\OG$ sweep of
Proposition~\ref{prop:ogmachine}, the attack cores of
Proposition~\ref{prop:cores}, the machine atlas and the ladder rungs ---
are reproducible from the scripts but are not accompanied by proof
certificates. The validity of the $\forall M$ schemas
rests on the hand proofs of Sections~\ref{sec:og}--\ref{sec:c3} and not on
any of these verdicts; their per-step machine audit
(Remark~\ref{rem:c3machine}) covers a finite range of scales only.

\section*{Acknowledgments and computational assistance}

I thank Jesse Geneson for comments on an earlier version of this paper,
which corrected two statements at the points where they apply: the balance
law (Lemma~\ref{lem:balance}), given there in a false symmetric form, and
the class-sink observation (Remark~\ref{rem:classsink}), asserted there
with an $r$-dependent cycle structure. His comments did not extend to a
verification of Theorem~\ref{thm:main} or of the chain leading to it.

This paper was prepared with machine assistance, which I disclose in full.
The language model Claude (Anthropic) was used for proof auditing, code
assistance, and editorial revision. The SAT solvers CaDiCaL and OR-Tools
CP-SAT were used for the computational experiments reported in
Section~\ref{sec:data} and Appendix~\ref{sec:compobs}, and for the
discovery work described there. The author takes full responsibility for
the contents of this paper, including every statement and proof and any
error that remains.

\appendix

\section{The chunk reduction}\label{sec:chunk}

\begin{sloppypar}
Every permutation of $\SA$ of order type $\omega$ decomposes into
consecutive finite segments (\emph{chunks}); conversely, a choice of finite
fibers and internal orders assembles into such a permutation. The exact
reformulation below is the coordinate system used in the elimination phase
of Appendix~\ref{sec:compobs}. As noted in the introduction, the reduction
of Section~\ref{sec:og} is proved without it, and nothing in this appendix
enters the proof of Theorem~\ref{thm:main}.
\end{sloppypar}

\begin{definition}
Let $S \subseteq \Z^+$ be infinite. A \emph{stage function} for $S$ is any
$s\colon S \to \N$ with finite fibers.
Given $s$ and a linear order on each fiber, let $T$ be the concatenation.
\end{definition}

\begin{theorem}[Chunk reduction; exact]\label{thm:chunk}
$T$ avoids monotone $3$-term APs if and only if:
\begin{enumerate}
\item[(A)] no AP triple $(x,y,z)$, $z=2y-x$, in $S$ has strictly monotone
stages ($s(x)<s(y)<s(z)$ or $s(z)<s(y)<s(x)$); and
\item[(B)] each fiber order contains the \emph{forced pairs}
\[
\begin{array}{ll}
s(x)<s(y)=s(z):\ z \prec y, &\quad s(z)<s(y)=s(x):\ x \prec y,\\
s(x)=s(y)<s(z):\ y \prec x, &\quad s(y)=s(z)<s(x):\ y \prec z,
\end{array}
\]
and is non-monotone on every within-fiber AP triple. (Triples whose endpoint
stages lie on one side of the middle's are automatically safe.)
\end{enumerate}
Hence $S$ is $3$-permutable if and only if some stage function for $S$,
together with a choice of fiber orders, satisfies
(A) and (B). The fiber problems are finite and mutually independent given
$s$. (Nothing in the statement or the proof uses any property of $S$; we
apply it to $S = \SA$ and, in Theorem~\ref{thm:degs}, to $S = \Z^+$.)
\end{theorem}

\begin{proof}
Along positions of $T$ the stage is non-decreasing, and within a stage the
fiber order gives the positions. Suppose $T$ avoids monotone $3$-APs. If
(A) failed at $(x,y,z)$ with $s(x)<s(y)<s(z)$, then $x,y,z$ occur in $T$ in
that position order with $x<y<z$ --- an increasing $3$-AP (the case
$s(z)<s(y)<s(x)$ gives a decreasing one). If a forced pair of (B) failed,
say $s(x)<s(y)=s(z)$ with $y \prec z$ in their fiber, then $x, y, z$ occur
in position order --- again an AP; the other three rows and the within-fiber
non-monotonicity requirement are identical. Conversely, suppose (A) and (B)
hold and let $(x,y,z)$, $z = 2y-x$, be placed monotonically in $T$, say
$\pos(x)<\pos(y)<\pos(z)$ (the decreasing case is symmetric). Then
$s(x) \le s(y) \le s(z)$. Strict inequalities throughout contradict (A);
$s(x)=s(y)<s(z)$ puts $x \prec y$ in one fiber, contradicting the row
$s(x)=s(y)<s(z) \Rightarrow y \prec x$ of (B); $s(x)<s(y)=s(z)$ puts
$y \prec z$, contradicting $z \prec y$; and $s(x)=s(y)=s(z)$ is a monotone
within-fiber triple, excluded by (B). Triples whose endpoint stages lie on
one side of the middle's stage can never be placed monotonically, since the
middle's position is then extremal among the three; this is the safety
clause. The full case table over all stage patterns and fiber orders has
additionally been machine-checked exhaustively
(\texttt{e96\_reduction\_check.py}, part P2).
\end{proof}

\begin{lemma}[Normalization: running-max chunking]\label{lem:normal}
Every $3$-AP-free permutation $\pi$ of $\SA$ (of order type $\omega$)
admits a valid scheme --- a stage function with finite fibers satisfying
(A)$\wedge$(B), whose concatenation is $\pi$ itself --- with
$s(v) \ge \mathrm{block}(v)/2$ for every $v$.
\end{lemma}

\begin{proof}
Set $s(v) = \max_{q \le \pos(v)} \mathrm{block}(\pi(q))/2$ (the running
maximum of half-block-indices). Along positions $s$ is non-decreasing, so
its fibers are consecutive segments of $\pi$ and the concatenation (with
the fiber orders induced by $\pi$) is $\pi$; Theorem~\ref{thm:chunk} then
gives (A)$\wedge$(B). Since $\mathrm{block}(v)/2$ enters the maximum at
$v$'s own position, $s(v) \ge \mathrm{block}(v)/2$. Each fiber is finite:
$\pi$ is onto $\SA$, so a value of block index $> 2\sigma$ occurs at some
finite position, after which the running maximum exceeds $\sigma$.
\end{proof}

\noindent The point of Lemma~\ref{lem:normal} is that non-negative
\emph{displacement} (Appendix~\ref{sec:compobs}) is a harmless
normalization, not an
automatic property: an arbitrary chunk decomposition can place a
high-block value in an early chunk. Restrictions of normalized schemes
stay normalized, which is what the divergence criterion there consumes.

The definitions above and Theorem~\ref{thm:chunk} hold verbatim with $\SA$
replaced by any countably infinite set $S \subseteq \Z^+$: no step of the
statement or the proof uses a property of $\SA$. Applying that formulation
to $S = \Z^+$ gives:

\begin{theorem}[DEGS via chunks]\label{thm:degs}
$\Z^+$ itself is not $3$-permutable.
\end{theorem}

\begin{proof}
Suppose $s\colon \Z^+ \to \N$ satisfies
(A) with finite fibers and put $\sigma = s(1)$. For every $y$ with
$s(y) > \sigma$, the triple $(1, y, 2y-1)$ forces $s(2y-1) \le s(y)$ by (A).
The orbit $y_0,\, y_{i+1} = 2y_i - 1$ is strictly increasing; only finitely
many of its members can have stage $\le \sigma$ (those fibers are finite),
after which its stages are non-increasing, hence eventually confined to one
finite fiber --- impossible for an infinite set.
\end{proof}

\noindent This recovers the Davis--Entringer--Graham--Simmons theorem (in its
monotone form; \cite{DEGS77} proves the stronger increasing-3-AP statement)
in the chunk coordinates and isolates what a permutable team must do: break every
such reflected orbit out of the team. The dyadic team $\SA$ breaks
single-reflector orbits in two steps (completions land in odd blocks); the
question is whether the multi-scale accounting below can be survived.

\begin{lemma}[Same block]\label{lem:sameblock}
In every AP triple of $\SA$ the middle $y$ and top $z$ lie in the same even
block. ($z<2y\le 2^{k+1}$ and block $k{+}1$ is odd.)
\end{lemma}

\begin{lemma}[Escape]\label{lem:escape}
If $y \ge 3\cdot 4^{s-1}$ (top quarter of block $2s$) and $x \le 4^{s-1}$,
then $2y-x > 4^{s}$: completions of (lower value, top-quarter value) pairs
land in the odd block above --- out of team.
\end{lemma}

\begin{lemma}[Bottom-half characterization]\label{lem:bottomhalf}
Deferring $D_k \subseteq$ block $k$ by a uniform delay (adding a constant to
$s$ on $D_k$) creates no violation of condition (A) if
$D_k \subseteq (2^{k-1}, 3\cdot 2^{k-2}]$: the witness midpoints
$y=(x+z)/2$ of a bottom-half $z \in D_k$ against an $\SA$-value $x$ in a
lower block (so $x \le 2^{k-2}$) satisfy
$2^{k-2} < y \le 2^{k-1}$, i.e.\ lie in the odd block below. Conversely a
deferred top-half $z$ has genuine in-team witnesses, forcing its witness
midpoints (bottom-quarter values) to defer with it.
\end{lemma}

\section{Computational observations and conjectures}\label{sec:compobs}

Everything in this appendix is a record of \emph{finite-scale computation}:
measurements, solver verdicts at particular horizons, and the conjectures
they suggest. None of it is used in the proof of Theorem~\ref{thm:main},
whose chain is Theorem~\ref{thm:ogred} together with
Theorem~\ref{thm:c3core}. We report it because it is what located the
order gadget and the two crowns $15, 16$, and because the conjectures it
raises seem to us the natural next questions. Statements here that are
quantified over all scales are explicitly labelled as conjectures; the
rest are verdicts at the horizons named, and carry no proof certificates
(see Section~\ref{sec:data}).

\subsection{Local realizability and the tower characterization}

\begin{definition}[Pure-complete systems]
For an $\SA$-block top $X$, the \emph{pure-complete-$X$ system} asks for an
arrangement of $\SA \cap [1, X]$ with no monotone 3-AP. (All completions of
in-range pairs lie in $(X, 2X]$, an odd block, or inside the range; hence this
is exactly the restriction of the infinite problem.)
\end{definition}

\begin{lemma}[Horizon collapse]
Deleting any set of values from a valid arrangement preserves validity; hence
auxiliary values above $X$ never aid the pure-complete-$X$ system.
\end{lemma}

\begin{theorem}[Tower characterization]
$\SA$ is permutable iff there exists $B\colon \N \to \N$ such that for every
block top $X$ the system [pure-complete-$X$ and every $v$ has at most $B(v)$
predecessors] is satisfiable.
\end{theorem}

\begin{proof}
($\Rightarrow$) restrict a permutation and take $B(v)$ its positions.
($\Leftarrow$) the restriction sets are finite and nested; K\H{o}nig's lemma
yields a consistent tower whose limit has each $v$ preceded by at most $B(v)$
values, hence order type $\omega$; all constraints are finitary and inherited.
\end{proof}

Computationally: pure-complete-$X$ is satisfiable for $X \le 1024$
(explicit witnesses in the repository), and the completion-pressure function
$g_X(L) = \min \max_{v \le L} \pos(v)$ satisfies $g_{256}(64) = 153$,
this value being optimal. The equality $g_{256}(64) = 153$ means that at
least $153 - 42 = 111$ of the $128$ values in $(128,256]$ precede the
last-placed member of $\SA \cap [1,64]$ (recall
$|\SA \cap [1,64]| = 42$, so at most $42$ of the first $153$ positions can
be occupied by members of $\SA \cap [1,64]$ itself). In the optimal
arrangements found it is the \emph{deferred} values that are structured: the
$16$ members of $(128,256]$ placed after that point are almost exactly a
single residue class modulo $8$ (fifteen of them lie in class $1$).
Divergence of $g_X(L)$ in $X$ for fixed $L$ would prove
non-permutability directly; our measurements did not exhibit divergence
(the subset-tolerance is near-total), and this route was abandoned in favour
of the one taken in Sections~\ref{sec:og}--\ref{sec:c3}. There is no
tension with Theorem~\ref{thm:main}: stabilization of $g_X(L)$ for each
fixed $L$ does not by itself produce a single bound $B$ as required by the
tower characterization, because the minimizing arrangements vary with $L$.

\subsection{Self-similar witnesses, fragility, and pumping obstructions}

Adding the exact scale-invariance $u \prec w \iff 4u \prec 4w$ to the
pure-complete systems remains satisfiable at $X = 256$ and $X = 1024$
(independently audited witnesses); however the audited $1024$-witness does not
extend to a self-similar $4096$-arrangement (solver verdict, uncertified),
so the self-similar tower also has dead branches.

We also record a calibration fact: requiring, in addition, that the $\pm 1$
neighbors (within the same block and team) of each forced completion also
precede the pair's later element (``radius-1 robustness'') is infeasible in
every setting we tested --- including plain intervals $[1, n]$ from $n = 8$,
where no partition is involved at all. We have no proof that this persists,
but within the tested range radius-1 robustness fails for reasons that have
nothing to do with the dyadic partition, so it does not discriminate between
partitions and is of no use as a satisfiability criterion; a pumping
argument would instead have to manage
rounding drift on the thin family of far-pair constraints only. The
radius-one experiments are \emph{evidence} against placement rules in which
position varies slowly with value; we neither formalize such a class of
rules here nor prove a theorem about it, and we draw no further inference
from the observation.

\subsection{The negative atlas}\label{sec:negative}

\begin{theorem}[Block-granular death]\label{thm:blockgranular}
No stage function whose fibers are unions of whole even blocks satisfies
(A)$\wedge$(B).
\end{theorem}

\begin{proof}
For an even $b \ge 6$ define the \emph{fatal core} $F_b$: the variables are
the relative order of the elements of the block $B_b$, and the constraints
are
\begin{enumerate}
\item[(a)] the forced pair $z \prec y$ for every AP triple $(x, y, z)$,
$z = 2y - x$, with $x \in \SA$ lying below $B_b$ and $y, z \in B_b$ --- these are the
instances of the row $s(x) < s(y) = s(z) \Rightarrow z \prec y$ of (B) that
a stage boundary below $B_b$ creates; and
\item[(b)] non-monotonicity on every AP triple lying inside $B_b$.
\end{enumerate}
Any fiber whose lowest block is $B_b$ contains $F_b$ as a subsystem, and a
supersystem of an unsatisfiable system is unsatisfiable.

$F_6$ is unsatisfiable: it has $51$ forced pairs and $240$ triples on the
$32$ elements of $B_6$, and is refuted by CaDiCaL in well under a second
(script \texttt{e59}, independently reproduced by \texttt{e67}; this is a
solver verdict on one finite instance and carries no proof certificate, cf.\
Section~\ref{sec:data}). The map $v \mapsto 4v$ then embeds $F_b$ into
$F_{b+2}$: it
preserves $\SA$, sends block $j$ to block $j{+}2$, and preserves AP triples;
so every $F_b$ with $b \ge 6$ even is unsatisfiable.
It remains to see that a block-granular scheme
has an \emph{ascending} fiber boundary above block $4$, i.e.\ some
$b \ge 6$ with $s(B_{b-2}) < s(B_b)$ (this is what puts the forced pairs of
$F_b$ into the system). If there were none, then
$s(B_4) \ge s(B_6) \ge s(B_8) \ge \cdots$ would be a non-increasing
sequence of natural numbers, hence eventually constant; infinitely many
blocks would then share a single fiber, contradicting the finiteness of
fibers.
\end{proof}

\begin{theorem}[$\times 4$ restriction; window death is permanent]
\label{thm:restriction}
For the window class $\{s : s(v)-\mathrm{block}(v)/2 \in [0,w]\}$:
feasibility at horizon $4N$ implies feasibility at horizon $N$ (restrict
along $v \mapsto 4v$ and shift stages by one). Consequently the machine
verdicts below are permanent (hold at all larger horizons), and an infinite
solution in a window class would imply feasibility at every finite horizon.
\end{theorem}

\noindent Machine atlas (CP-SAT and CDCL cross-checks; solver logs and
witnesses in the repository, but no proof certificates for the negative
verdicts): bottom-half relief of depth $1$ and depth $\le 2$ infeasible at
$N=256$ (and $1024$); \textbf{window-$2$ infeasible at $N=1024$} (CP-SAT, 746s; independently CDCL/unary-ladder, 1760s) --- hence, by
Theorem~\ref{thm:restriction}, infeasible at every horizon
$\ge 1024$, so there is no infinite window-$2$ solution. (At the smaller
horizon $N=256$ window~$2$ is still feasible, but only via whole-block
merges, of minimum total displacement $41$; the merged skeleton is
$C_s = 2^{s}\cdot[3\cdot 2^{s-2}, 2^{s}]$, the top-quarter multiples of
$2^{s}$ --- precisely the set singled out by Lemma~\ref{lem:escape}.)

\medskip
Among the classes of stage function we were able to search, the atlas
leaves one shape of candidate for a YES: stage functions
with unbounded displacement (chunks mixing stragglers of unboundedly many
blocks). Closing the remaining gap --- between
Theorems~\ref{thm:blockgranular}/\ref{thm:restriction} and general
finite-fiber schemes --- by the per-block deferral calculus of
Lemma~\ref{lem:bottomhalf} was the state of the elimination phase before
Section~\ref{sec:og}, where the gap is closed instead by the order-gadget
route, which does not use the chunk calculus at all. The same analysis
applies to $\SB{}$ by the evident mirror symmetry (checked by machine at
$4096$).

\subsection{The displacement ladder}\label{sec:ladder}

For a stage function $s$ write $\delta(v) = s(v) - \mathrm{block}(v)/2$
(\emph{displacement}); a scheme is \emph{normalized} if $\delta \ge 0$
everywhere, which by Lemma~\ref{lem:normal} is no loss of generality for
schemes induced by genuine permutations. Define
\[
L(m) \;=\; \min_{\text{normalized schemes at horizon } 4^m}\
\max_{v \in \SA,\ v \le 16} \delta(v),
\]
the minimum over all normalized finite-fiber schemes satisfying
(A)$\wedge$(B) at horizon $4^m$ (no bound on any other value's
displacement); the maximum is over the ten values
$\SA \cap [1,16] = \{3,4\} \cup \{9, \dots, 16\}$.

\begin{proposition}[Ladder rungs; machine-checked]
$L(3) \le 1$; \ $L(4) \ge 2$ (infeasibility of the cap $\delta \le 1$ on
$\SA \cap [1,16]$ at $N=256$, verified independently by CP-SAT, 252s, and
CDCL with a unary-ladder encoding, 15s); \ $L(5) = 3$ (lower bound: CDCL,
2304s, a single run; upper bound: a window-3 witness at $N=1024$, 995s ---
which also gives $L(4) \le 3$, since $L$ is non-decreasing in $m$).
Moreover the minimal infeasible cap-coalition at $N=256$ is exactly
$\{15, 16\}$ (deletion-minimal): every scheme has $\delta(15) \ge 2$ or
$\delta(16) \ge 2$ --- the crowns of the two competing class towers
($15 = 1111_2$, apex of the $\equiv -1$ skeleton; $16 = 10000_2$, the block
top, apex of the $\equiv 0$ skeleton).
\end{proposition}

\begin{theorem}[Divergence criterion]\label{thm:divergence}
If $L(m) \to \infty$ then $\SA$ is not $3$-permutable.
\end{theorem}

\begin{proof}
An
infinite valid permutation induces (Lemma~\ref{lem:normal}) a
\emph{normalized} valid scheme, and by restriction a normalized valid
scheme at every horizon $4^m$ (completions of values $\le 4^m$ that lie
beyond the horizon fall in the odd block $(4^m, 2\cdot 4^m]$, out of
team), and the ten displacements $\delta(v)$, $v \in \SA \cap [1,16]$,
are constants of the infinite
scheme. For $m$ with $L(m) > \max_{v \in \SA,\, v\le 16}\delta(v)$ this is a
contradiction, by the exactness of Theorem~\ref{thm:chunk}.
\end{proof}

The conjectured growth $L(m) = m-2$ matches the window-death schedule
(window $w$ dying at horizon $4^{w+3}$) and the observed per-block descent
of displacement in minimal solutions ($\delta$ profile $(m{-}2, m{-}3,
\dots, 1, 0)$ across blocks $4, 6, \dots, 2m$). The missing step is a
self-propagating squeeze $L(m{+}1) \ge L(m) + 1$; the $\times 4$
restriction alone yields only the diagonal transfer
$L_{b+1}(m{+}1) \ge L_b(m)$ for the capped set $[1, 4^b]$.

\end{document}